\documentclass[11pt]{article}

\usepackage{amssymb,amsmath,amsfonts,amsthm}
\usepackage{latexsym}
\usepackage{graphics}
\usepackage{indentfirst}
\usepackage{tikz}
\usepackage{mathtools}
\usepackage{hyperref}
\hypersetup{
    colorlinks=true,
    linkcolor=blue,
    filecolor=magenta,      
    urlcolor=red,
    citecolor=cyan,
}
\usepackage[spacing=true,kerning=true,babel=true,tracking=true]{microtype}
\usepackage[shortlabels]{enumitem}
\usepackage{array}

\usepackage{lmodern}

\usepackage[
    sortcites,
    backend=biber, style=alphabetic, sorting=nyt, maxnames=100,backref=true]{biblatex}
\usetikzlibrary{shapes,arrows}

\allowdisplaybreaks

\usepackage[left=1in,right=1in,top=1in,bottom=1in,bindingoffset=0cm]{geometry}

\newtheoremstyle{slthm}%
{}{}%
{\slshape}{}%
{\bfseries}{\bfseries.}%
{ }%
{\thmname{#1}\thmnumber{ #2}\thmnote{ \normalfont{}({#3})}}
\theoremstyle{slthm}

\newtheorem{innercustomthm}{Theorem}
\newenvironment{customthm}[1]
  {\renewcommand\theinnercustomthm{#1}\innercustomthm}
  {\endinnercustomthm}
\newtheorem{innercustompro}{Proposition}
\newenvironment{custompro}[1]
  {\renewcommand\theinnercustompro{#1}\innercustompro}
  {\endinnercustompro}

\newtheorem{innercustomcor}{Corollary}
\newenvironment{customcor}[1]
  {\renewcommand\theinnercustomcor{#1}\innercustomcor}
  {\endinnercustomcor}

\newtheorem*{thm*}{Theorem}
\newtheorem{thm}{Theorem}[section]
\newtheorem{lem}[thm]{Lemma}
\newtheorem{pro}[thm]{Proposition}
\newtheorem{obs}[thm]{Observation}
\newtheorem{cor}[thm]{Corollary}
\newtheorem{cl}[thm]{Claim}
\newtheorem{ques}[thm]{Question}

\theoremstyle{definition}
\newtheorem{proc}[thm]{Procedure}
\newtheorem{conj}[thm]{Conjecture}

\newcommand{\N}{\mathbb{N}}

\newcommand{\E}{\mathbb{E}}

\newcommand{\R}{\mathbb{R}}

\newcommand{\Var}{\mathrm{Var}}
\newcommand{\Cov}{\mathrm{Cov}}

\renewcommand{\leq}{\leqslant}
\renewcommand{\geq}{\geqslant}
\renewcommand{\epsilon}{\varepsilon}
\renewcommand{\phi}{\varphi}
\newcommand{\defeq}{=}

\numberwithin{equation}{section}

\title{DP-Coloring  of Graphs from Random Covers}
\author{Anton Bernshteyn\thanks{Department of Mathematics, University of California, Los Angeles, CA, USA. E-mail: {\tt bernshteyn@math.ucla.edu}. Research of this author is partially supported by the NSF grant DMS-2045412 and the NSF CAREER grant DMS-2239187.} \and
Daniel Dominik\thanks{Department of Applied Mathematics, Illinois Institute of Technology, Chicago, IL, USA. E-mail: {\tt ddominik@hawk.iit.edu}} \and
Hemanshu Kaul\thanks{Department of Applied Mathematics, Illinois Institute of Technology, Chicago, IL, USA. E-mail: {\tt kaul@iit.edu}} \and
Jeffrey A. Mudrock\thanks{Department of Mathematics and Statistics, University of South Alabama, Mobile, AL, USA. E-mail: {\tt mudrock@southalabama.edu}}}
\date{\today}
 
\begin{document}
\maketitle 
\begin{abstract}
	\noindent DP-coloring (also called correspondence coloring) of graphs is a generalization of list coloring that has been widely studied since its introduction by Dvo\v{r}\'{a}k and Postle in $2015$. Intuitively, DP-coloring generalizes list coloring by allowing the colors that are identified as the same to vary from edge to edge. Formally, DP-coloring of a graph $G$ is equivalent to an independent transversal in an auxiliary structure called a DP-cover of $G$.  
    In this paper, we introduce the notion of random DP-covers and study the behavior of DP-coloring from such random covers. We prove a series of results about the probability that a graph is or is not DP-colorable from a random cover.  These results support the following threshold behavior on random $k$-fold DP-covers as $\rho\to\infty$ where $\rho$ is the maximum density of a graph: graphs are non-DP-colorable with high probability when $k$ is sufficiently smaller than $\rho/\ln\rho$, and graphs are DP-colorable with high probability when $k$ is sufficiently larger than $\rho/\ln\rho$.  Our results depend on $\rho$ growing fast enough and imply a sharp threshold for dense enough graphs.  For sparser graphs, we analyze DP-colorability in terms of degeneracy.  We also prove fractional DP-coloring analogs to these results.

	\medskip



 \end{abstract}

 \section{Introduction} \label{intro}

\subsection{Basic terminology and notation}

    \noindent All graphs in this paper are finite and simple.  Unless otherwise noted we follow terminology from West~\cite{W}.  We will use $\N$ for the set $\{0,1,2,\ldots\}$ of natural numbers, $2^A$ for the power set of a set $A$, and $[k]$ for $\{1,\,\ldots,\,k\}$ with $[0]=\emptyset$.  Also, for $a$, $b\in\N$ with $a\leq b$, $[a:b]$ is the set $\{a,\ldots,b\}$. We use $K_n$ for the {complete graph} on $n$ vertices and $K_{m \times n}$ 
    for 
    the complete $m$-partite graph with parts of size $n$.  For a graph $G$ and $t\in\N$, $tG$ is the disjoint union of $t$ copies of $G$.

	Suppose $G=(V(G),E(G))$ is a graph and $v\in V(G)$ is a vertex. The \emph{neighborhood} of $v$ in $G$ is the set of all vertices adjacent to $v$, and it is denoted by $N_G(v)$. The \emph{degree} of $v$ is $|N_G(v)|$, and it is denoted by $\deg(v)$. 
 For two disjoint sets of vertices $A$, $B\subseteq V(G)$, we will use $E_G(A,B)$ to denote the set of edges with one endpoint in $A$ and the other in $B$.

    The \emph{density} of a nonempty graph $G$, denoted by $d(G)$, is $|E(G)|/|V(G)|$.  The \emph{maximum density} of a nonempty graph $G$, denoted by $\rho(G)$, is $\max_{G'} d(G')$, where the maximum is taken over all nonempty subgraphs $G'$ of $G$. 
    
	A graph $G$ is said to be \emph{$d$-degenerate} if there exists some ordering of the vertices in $V(G)$ such that each vertex has at most $d$ neighbors among the preceding vertices.  The \emph{degeneracy of a graph $G$} is the smallest $d \in \N$ such that $G$ is $d$-degenerate. Note that the degeneracy $d$ of $G$ satisfies the bounds $\rho(G) \leq d \leq 2 \rho(G)$.

	\subsection{Graph coloring, list coloring, and DP-coloring}

		\noindent In classical vertex coloring, given a graph $G$, we assign to each vertex $v\in V(G)$ a color from $\N$.  More precisely, by a \emph{coloring} of $G$ we mean a function $\phi \colon V(G)\to\N$.  A \emph{$k$-coloring} is a coloring where 
        $\phi(v) \in [k]$ for all $v \in V(G)$.  We say that a coloring $\phi \colon V(G)\to \N$ is \emph{proper} if for every $uv\in E(G)$, $\phi(u)\neq\phi(v)$.  We say that $G$ is \emph{$k$-colorable} if it has a proper $k$-coloring.  The \emph{chromatic number} of $G$, $\chi(G)$, is the smallest $k$ such that $G$ is $k$-colorable.

		List coloring is a generalization of vertex coloring that was introduced in the 1970s independently by Vizing \cite{V} and Erd\H{o}s, Rubin, and Taylor \cite{ERT}.  For a graph $G$, a \emph{list assignment} for $G$ is a function $L \colon V(G)\to 2^{\N}$; intuitively, $L$ maps each vertex $v\in V(G)$ to a list of allowable colors $L(v)\subseteq \N$.  An \emph{$L$-coloring} of $G$ is a coloring $\phi$ of $G$ such that $\phi(v)\in L(v)$ for each $v\in V(G)$.  A \emph{proper $L$-coloring} of $G$ is an $L$-coloring of $G$ that is a proper coloring.  A \emph{$k$-assignment} for $G$ is a list assignment $L$ for $G$ such that $|L(v)|=k$ for all $v\in V(G)$.  We say $G$ is \emph{$k$-choosable} if for every $k$-assignment $L$ for $G$, $G$ has a proper $L$-coloring.  The \emph{list chromatic number} of $G$, denoted by $\chi_{\ell}(G)$, is the smallest $k$ such that $G$ is $k$-choosable.  Clearly, $\chi(G)\leq\chi_\ell(G)$.

		The concept of DP-coloring was first put forward in 2015 by Dvo\v{r}\'{a}k and Postle~\cite{DP} under the name \emph{correspondence coloring}.  Intuitively, DP-coloring generalizes list coloring by allowing the colors that are identified as the same to vary from edge to edge.  Formally, for a graph $G$, a \emph{DP-cover} (or simply a \emph{cover}) of $G$ is an ordered pair $\mathcal{H}=(L,H)$, where $H$ is a graph and $L \colon V(G)\to 2^{V(H)}$ is a function satisfying the following conditions: 
		\begin{itemize}
			\item $\{L(v) : v \in V(G)\}$ is a partition of $V(H)$ into $|V(G)|$ independent sets, 

			\item for every pair of adjacent vertices $u$, $v\in V(G)$, the edges in $E_H\left(L(u),L(v)\right)$ form a matching (not necessarily perfect and possibly empty), and

			\item $\displaystyle E(H) = \bigcup_{uv \in E(G)} E_{H}(L(u),L(v)).$
		\end{itemize}
        Note that some definitions of a DP-cover require each $L(v)$ to be a clique (see, e.g., \cite{BKP}), but our definition requires each $L(v)$ to be an independent set of vertices in $H$.  We will refer to $L(v)$ as the \emph{list} of $v$.  

		Suppose $\mathcal{H}=(L,H)$ is a cover of a graph $G$.  A \emph{transversal} of $\mathcal{H}$ is a set of vertices $T\subseteq V(H)$ containing exactly one vertex from each list $L(v)$. 
        A transversal $T$ is said to be \emph{independent} if $uv\not\in E(H)$ for all $u$, $v\in T$, i.e., if $T$ is an independent set in $H$.  If $\mathcal{H}$ has an independent transversal $T$, then $T$ is said to be a \emph{proper $\mathcal{H}$-coloring} of $G$, and $G$ is said to be \emph{$\mathcal{H}$-colorable}.

		A \emph{$k$-fold cover} of $G$ is a cover $\mathcal{H}=(L,H)$ such that $|L(v)|=k$ for all $v\in V(G)$.  We will use the term \emph{cover size} to refer to the value of $k$ in a $k$-fold cover.  The \emph{DP-chromatic number} of a graph $G$, $\chi_{DP}(G)$, is the smallest $k \in \N$ such that $G$ is $\mathcal{H}$-colorable for every $k$-fold cover $\mathcal{H}$ of $G$. 
        Since for any $k$-assignment $L$ for $G$, there exists a $k$-fold cover $\mathcal{H}$ of $G$ such that $G$ is $L$-colorable if and only if it is $\mathcal{H}$-colorable~\cite{DP}, we know that
		\begin{equation*}
			\chi(G)\,\leq\,\chi_{\ell}(G)\,\leq\,\chi_{DP}(G).
		\end{equation*}
		Many classical upper bounds on $\chi_\ell(G)$ hold for $\chi_{DP}(G)$ as well \cite{DP, B, B3}, but there are also some differences \cite{BKdiff}.  For example, the first named author~\cite{B} showed that for a graph $G$ with average degree $d$, $\chi_{DP}(G)=\Omega(d/\ln d)$. On the other hand, by a celebrated result of Alon~\cite{A}, such graphs satisfy $\chi_\ell(G)=\Omega(\ln d)$, and this bound is in general best possible. 

		A \emph{full cover} of $G$ is a cover $\mathcal{H}=(L,H)$ such that for any $uv\in E(G)$, the matching between $L(u)$ and $L(v)$ is perfect.  It follows that a full cover of a connected graph $G$ must be a $k$-fold cover for some $k\in\N$. It is clear that in the definition of the DP-chromatic number, it is enough to only consider full $k$-fold covers.


	\subsection{Random DP-covers}\label{subsec:randomCover}

		\noindent The DP-chromatic number, similarly to other variants of the chromatic number, captures the extremal behavior of a graph $G$ with respect to DP-coloring. This extremal perspective looks for the value $k_0 \in \N$ such that $G$ is $\mathcal{H}$-colorable for every $k$-fold cover $\mathcal{H}$ with $k\geq k_0$, and $G$ is not $\mathcal{H}$-colorable for some $k$-fold cover $\mathcal{H}$ for each $k<k_0$.  However, for a given graph $G$, one can also ask what DP-coloring behavior $G$ exhibits with a high (or low) proportion of DP-covers of a given cover size.  This notion of high (or low) proportion of DP-covers can be naturally formalized by considering a probability distribution on the set of all such covers.  In this paper, we initiate this study by considering full DP-covers generated uniformly at random, and asking the natural probabilistic questions of the likelihood that the cover does or does not have an independent transversal, and whether a graph shows asymptotically almost sure transition in the behavior of its DP-coloring over these random covers.

		Similar probabilistic questions have also been studied for list coloring of graphs.  The list assignments of a given graph $G$ are generated uniformly at random from a palette of given colors, and the primary question is whether there is a threshold size of the assignments that shows a transition in the list colorability of the graph (parameterized by either the number of vertices or the chromatic number of the graph). Krivelevich and Nachmias~\cite{KN1, KN2} initiated the study of this topic in 2005 and considered it specifically for complete bipartite graphs and powers of cycles.  Several follow-up papers were written by Casselgren~\cite{C1, C2, C3, C4} and Casselgren and H\"aggkvist \cite{CH} on various families of graphs including complete graphs, complete multipartite graphs, graphs with bounded degree, etc. The concept of colorings from random list assignments---under the name \emph{palette sparsification}---has recently found applications in the design of sublinear algorithms in the work of Assadi, Chen, and Khanna \cite{ACK}.

		As discussed in the paper by Dvo\v{r}\'{a}k and Postle~\cite{DP}, a full DP-cover of $G$ is equivalent to the previously studied notion of a lift (or a covering graph) of $G$ (see Godsil and Royle~\cite[\S6.8]{GR} and discussion with further references in~\cite{AL}).  The notion of random $k$-lifts was introduced by Amit and Linial~\cite{AL}.  This work, and the large body of research following it~\cite{AL2, ALM, LR}, studied random $k$-lifts as a random graph model.  Their purpose was the study of the properties of random $k$-lifts of a fixed graph $G$ as $k\to\infty$.  In this paper, we study the probability that a randomly selected cover has an independent transversal.  This focus on finding independent transversals is different from the previous work on random $k$-lifts.

        Let us describe our random model formally. Suppose $G$ is a graph with $V(G)=\{v_1,\ldots,v_n\}$ and $E(G)=\{e_1,\ldots,e_m\}$.  For some $k\in\N$, let $L(v)\defeq\{(v,i):i\in[k]\}$ for each $v\in V(G)$.  Let $S_k$ be the set of all permutations of $[k]$.  For each ${\boldsymbol\sigma}=(\sigma_1,\ldots,\sigma_m)\in S_k^m$, let $\mathcal{H}_{\boldsymbol\sigma}=(L,H_{\boldsymbol\sigma})$ be the full $k$-fold cover of $G$, where $H_{\boldsymbol\sigma}$ has the following edges.  For each $j\in[m]$, consider the edge $e_j$.  Suppose $e_j=v_qv_r$ with $q<r$, and define 
        \[
            E_{H_{\boldsymbol\sigma}}\left(L(v_q),L(v_r)\right) \,\defeq\, \{(v_q,i)(v_r,\sigma_j(i)) \,:\,i\in[k]\}.
        \]
        Let $\Omega_{G,k}\defeq\{\mathcal{H}_{\boldsymbol\sigma}:{\boldsymbol\sigma}\in S_k^m\}$. Note that $|\Omega_{G,k}|=(k!)^{m}$. 
		Let $\mathcal{H}(G,k)$ be an element of $\Omega_{G,k}$ chosen uniformly at random. We will refer to $\mathcal{H}(G,k)$ as a \emph{random $k$-fold cover of $G$}.

		Equivalently, a random $k$-fold cover of a graph $G$, $\mathcal{H}(G,k)=(L,H)$, can be constructed in the following way.  The mapping $L$ is as described above.  For each $v\in V(G)$, $L(v)$ is an independent set in $H$.  For each edge $e\in E(G)$, where $e=v_qv_r$, the edge set $E_H\left(L(v_q),L(v_r)\right)$ takes on one of the $k!$ perfect matchings between $L(v_q)$ and $L(v_r)$ uniformly at random. 
  
        Here we study the probability that a random cover of a graph has an independent transversal.  When $\mathbb{P}(\text{$G$ is $\mathcal{H}(G,k)$-colorable}) = p$, we say that $G$ is \emph{$k$-DP-colorable with probability $p$}.  

	\subsection{Outline of the main results} 

        \subsubsection{DP-colorability and density}

		\noindent 
        Table~\ref{tab:compilation} compiles our main results.  Given $\epsilon > 0$ and a graph $G$ with a large enough number $n$ of vertices whose maximum density is bounded below by the quantity in the first column, the third column gives the probability that $G$ is $k$-DP-colorable provided $k$ satisfies the inequality in the second column. 

        {
        \begin{table}[h]
			\centering
			\begin{tabular}{|c|c|>{\centering\arraybackslash}p{2.7cm}|c|}
				\hline
				Density lower bound&Cover size&Probability of $k$-DP-colorability&Reference\\
				\hline
				$\exp(e/\epsilon)$&$\displaystyle k\leq\frac{\rho(G)}{\ln\rho(G)}$&$\leq \epsilon$&Prop.~\ref{prop:first}\\
				&&&\\
				$n^{1-s}$ for $s\in[0,1/3)$&$\displaystyle k\geq(1+\epsilon)\left(1+\frac{s}{1-2s}\right)\frac{\rho(G)}{\ln\rho(G)}$&$\geq 1 - \epsilon$&Thm.~\ref{thm:slidingsharpness}\\
				&&&\\
				$\displaystyle\ln^{2/\epsilon}n$&$\displaystyle k\geq(1+\epsilon)\frac{2\rho(G)}{\ln\rho(G)}$&$\geq 1 - \epsilon$&Thm.~\ref{Degeneracy}\\
                &&&\\
				No lower bound & $k > 2\rho(G)$ &$1$&\cite{DP}\\
				\hline
			\end{tabular}
			\caption{Probability of $k$-DP-colorability depending on the maximum density of $G$. 
            }
			\label{tab:compilation}
		\end{table}
        }


        Our work shows that, for sufficiently dense graphs $G$, the probability $G$ is $k$-DP-colorable exhibits a transition when $k$ is close to $\rho(G)/\ln\rho(G)$. To begin with, we note that if $k \leq \rho(G)/\ln\rho(G)$ and $\rho(G)$ is not too small, the probability that $G$ is $k$-DP-colorable is close to $0$: 

        \begin{pro}\label{prop:first}
            Let $\epsilon>0$ and let $G$ be a nonempty graph with 
            $\rho(G)\geq \exp(e/\epsilon)$. If $1\leq k\leq\rho(G)/\ln\rho(G)$, then $G$ is $k$-DP-colorable with probability at most $\epsilon$.
            %
            %
		\end{pro}

        Proposition~\ref{prop:first} is proved via a simple application of the First Moment Method. It is a straightforward modification of \cite[Theorem 1.6]{B}, but, for completeness, we include a self-contained proof in \S\ref{subsec:FM}. We remark that the bound $\rho(G) \geq \exp(e/\epsilon)$ in Proposition~\ref{prop:first} is rather crude and can be sharpened by more careful calculations. However, this bound is sufficient for our purposes. In particular, note that if $(G_\lambda)_{\lambda \in \N}$ is a sequence of graphs such that $\rho(G_\lambda) \to \infty$ as $\lambda \to \infty$, then the bound $\rho(G_\lambda) \geq \exp(e/\epsilon)$ holds for any fixed $\epsilon > 0$ and all large enough $\lambda$.  
        
        
        Next we show that if $G$ is fairly dense (of density greater than $n^{2/3}$) and $k$ exceeds $\rho(G)/\ln \rho(G)$ by an appropriate constant factor, then $G$ is $k$-DP-colorable with probability close to $1$: 

        \begin{thm}\label{thm:slidingsharpness}
            For all $\epsilon > 0$ and $s \in [0, 1/3)$, there is $n_0 \in \N$ such that the following holds. Suppose $G$ is a graph with $n \geq n_0$ vertices such that $\rho(G)\geq n^{1-s}$, and
            \begin{equation}\label{eq:2nd}
                k \,\geq\, \left(1+\epsilon\right)\left(1+\frac{s}{1-2s}\right)\frac{\rho(G)}{\ln \rho(G)}.
            \end{equation}
            Then $G$ is $k$-DP-colorable with probability at least $1-\epsilon$.
            %
            %
		\end{thm}

        Theorem~\ref{thm:slidingsharpness} is proved in \S\ref{sec:kCol} by applying the Second Moment Method to the number of independent transversals in a random cover of $G$.  Notice that when the density of $G$ is at least $n^{1-o(1)}$, the factor in front of $\rho(G)/\ln \rho(G)$ in \eqref{eq:2nd} is very close to $1$ (which matches the bound in Proposition~\ref{prop:first}). 
        On the other hand, when $\rho(G) \geq n^{2/3 + o(1)}$, the factor in front of $\rho(G)/\ln\rho(G)$ approaches $2$. Our next result shows that with the factor fixed at $2$, the conclusion of Theorem~\ref{thm:slidingsharpness} remains true all the way down to $\rho(G) \geq \ln^{\omega(1)} n$. More precisely, we establish the following lower bound on the probability that $G$ is $k$-DP-colorable in terms of the relationship between $k$ and the degeneracy of $G$: 

        \begin{thm}\label{Degeneracy}
            For all $\epsilon \in (0,1/2)$, there is $n_0\in\N$ such that the following holds.  Let $G$ be a graph with $n\geq n_0$ vertices and degeneracy $d$ such that $d\geq \ln^{2/\epsilon}n$ and let $k\geq(1+\epsilon)d / \ln d$.  Then $G$ is $k$-DP-colorable with probability at least $1-\epsilon$.
            %
		\end{thm}

        Theorem~\ref{Degeneracy} is proved in \S\ref{sec:degeneracy} by analyzing a greedy algorithm for constructing an independent transversal in a random $k$-fold cover. The key tool we employ is a form of the Chernoff--Hoeffding bound for negatively correlated Bernoulli random variables due to Panconesi and Srinivasan~\cite{PS}. 
        Since every graph $G$ is $2\rho(G)$-degenerate, the lower bound on $k$ in Theorem~\ref{Degeneracy} is implied by $k\geq(1+\epsilon)2\rho(G)/\ln \rho(G)$, which is roughly a factor of $2$ away from the bound in Proposition~\ref{prop:first}. We conjecture that the factor of $2$ is not needed: 
        \begin{conj}\label{conj:one_is_right}
            For all $\epsilon > 0$, there exist 
            $C>0$ and $n_0 \in \N$ such that the following holds. Suppose $G$ is a graph with $n \geq n_0$ vertices such that $\rho(G)\geq \ln^{C}n$, and 
            \begin{equation*}
                k \,\geq\, \left(1+\epsilon\right)\frac{\rho(G)}{\ln \rho(G)}.
            \end{equation*}
            Then $G$ is $k$-DP-colorable with probability at least $1-\epsilon$.
        \end{conj}
        
        
        The polylogarithmic lower bound on $\rho(G)$ in Theorem~\ref{Degeneracy} and Conjecture~\ref{conj:one_is_right} is unavoidable, as the following proposition, proved in \S\ref{sec:sparse}, demonstrates:



        \begin{pro}\label{pro:sparse}
            For any $\epsilon > 0$ and $n_0\in\N$, there is a graph $G$ with $n \geq n_0$ vertices such that $\rho(G) \geq \left( \ln n / \ln\ln n\right)^{1/3}$ but, for every $k\leq 2\rho(G)$, $G$ is $k$-DP-colorable with probability less than $\epsilon$.
                
		\end{pro}

        Therefore, we cannot hope to get results similar to 
        Theorems~\ref{thm:slidingsharpness} and \ref{Degeneracy} for very sparse graphs. 
        Note that the bound $k \leq 2\rho(G)$ in Proposition~\ref{pro:sparse} is optimal: if $k > 2\rho(G)$, then $k$ must be strictly greater than the degeneracy of $G$, and thus $G$ is $k$-DP-colorable (with probability $1$) \cite{DP}.

        It follows from \cite[Theorem 1.3]{B3} that for each $\epsilon > 0$, there is $C_\epsilon>0$ such that every triangle-free regular graph $G$ with $\rho(G) \geq C_\epsilon$ satisfies $\chi_{DP}(G) \leq (1 + \epsilon) 2 \rho(G)/\ln \rho(G)$, and hence it is $k$-DP-colorable (with probability $1$) for all $k \geq (1+\epsilon) 2\rho(G)/\ln \rho(G)$. Together with Proposition~\ref{pro:sparse}, this shows that for very sparse graphs, it is impossible to determine up to a constant factor where the probability of DP-colorability transitions from approximately $0$ to approximately $1$ based on the maximum density alone.

    \subsubsection{Threshold functions}

        \noindent We can use the above results to answer questions about the asymptotic behavior of sequences of graphs.  Consider a sequence of graphs $\mathcal{G}=(G_\lambda)_{\lambda\in\N}$ and a sequence of integers $\kappa=(k_\lambda)_{\lambda\in\N}$.  We say that \emph{$\mathcal{G}$ is $\kappa$-DP-colorable with high probability}, or \emph{w.h.p.}, if
            \[
                \lim_{\lambda \to \infty} \mathbb{P}(\text{$G_\lambda$ is $\mathcal{H}(G_\lambda,k_\lambda)$-colorable}) \,=\, 1.
            \]
        Similarly, we say that \emph{$\mathcal{G}$ is non-$\kappa$-DP-colorable w.h.p.}~if \[
                \lim_{\lambda \to \infty} \mathbb{P}(\text{$G_\lambda$ is $\mathcal{H}(G_\lambda,k_\lambda)$-colorable}) \,=\, 0.
            \]
		A function $t_{\mathcal{G}} \colon \N \to \R$ is called a \emph{DP-threshold function for $\mathcal{G}$} if it satisfies the following two conditions:  if $k_\lambda=o(t_{\mathcal{G}}(\lambda))$, then $\mathcal{G}$ is non-$\kappa$-DP-colorable w.h.p., while  if $t_{\mathcal{G}}(\lambda)=o(k_\lambda)$, then $\mathcal{G}$ is $\kappa$-DP-colorable w.h.p. (Here the $o(\cdot)$ notation is used with respect to $\lambda \to \infty$.) Similarly, a function $t_{\mathcal{G}}$ is said to be a \emph{sharp DP-threshold function for $\mathcal{G}$} if it satisfies the following two conditions:  for any $\epsilon>0$,  $\mathcal{G}$ is non-$\kappa$-DP-colorable w.h.p.~when $k_\lambda\leq(1-\epsilon)t_{\mathcal{G}}(\lambda)$ for all large enough $\lambda$, and it is $\kappa$-DP-colorable w.h.p.~when $k_\lambda\geq(1+\epsilon)t_{\mathcal{G}}(\lambda)$ for all large enough $\lambda$.

        We can use Theorem~\ref{thm:slidingsharpness} to show that any sequence of graphs $\mathcal{G}=(G_\lambda)_{\lambda\in\N}$ whose densities are bounded below by $|V(G_\lambda)|^{1-o(1)}$ has $t_{\mathcal{G}}(\lambda)\defeq\rho(G_\lambda)/\ln(\rho(G_\lambda))$ as a sharp DP-threshold function, while for graph sequences with $\rho(G_\lambda) \geq \ln^{\omega(1)} n$, the same $t_\mathcal{G}$ is a DP-threshold function (but we do not know whether it is sharp).  This result is proved in \S\ref{subsection:Threshold} as Theorem~\ref{thm:threshold}. 
        The following two corollaries are special cases: 

		\begin{cor}\label{cor:Kn}
			For $\mathcal{G}=(K_n)_{n\in\N}$, the sequence of complete graphs, $t_{\mathcal{G}}(n)\defeq n/(2\ln n)$ is a sharp DP-threshold function.
		\end{cor}

		\begin{cor}\label{cor:Knn...}
			For $\mathcal{G}=(K_{m \times n})_{n\in\N}$ with constant $m\geq2$, the sequence of complete $m$-partite graphs with $n$ vertices in each part, $t_{\mathcal{G}}(n)\defeq (m-1)n/(2\ln n)$ is a sharp DP-threshold function.
		\end{cor}

        The existence of a (not necessarily sharp) DP-threshold function of order $\Theta(n/\ln n)$ for the sequence of complete graphs was recently proved 
        by Dvo\v{r}\'ak and Yepremyan using different methods \cite[Theorem 1.3]{DY}.

        From Theorems~\ref{thm:slidingsharpness} and~\ref{Degeneracy}, we know that sequences of graphs with at least polylogarithmic densities have DP-threshold functions, while very dense graph sequences have sharp DP-threshold functions.  Proposition~\ref{pro:sparse} shows that there is more going on, and leads to the following question.

        \begin{ques}\label{ques:1}
				\normalfont Under what conditions on $\mathcal{G}=(G_\lambda)_{\lambda\in\N}$ will $t_{\mathcal{G}}(\lambda)=\rho(G_\lambda)/\ln\rho(G_\lambda)$ be a DP-threshold function or a sharp DP-threshold function for $\mathcal{G}$?
		\end{ques}


    \subsubsection{Fractional DP-coloring}

		\noindent We now extend our analysis of DP-colorability with respect to random covers to 
        fractional DP-coloring, which was introduced by Kostochka, Zhu, and the first named author in \cite{BKZ}.  For $a$, $b\in\N$, a \emph{b-fold transversal} of some $a$-fold cover $\mathcal{H}=(L,H)$ of $G$ is a set $T \subseteq V(H)$ such that for every $v\in V(G)$, $\left|L(v)\cap T\right|=b$.  An \emph{independent $b$-fold transversal} is a $b$-fold transversal that is an independent set in $H$ (here we rely on our convention that the sets $L(v)$ for $v \in V(G)$ are independent).  If $\mathcal{H}$ has an independent $b$-fold transversal, we say that $G$ is \emph{$(\mathcal{H},b)$-colorable}.  
        For $a$, $b\in\N$ with $a\geq b$, we say that a graph $G$ is \emph{$(a,b)$-DP-colorable} if for every $a$-fold cover $\mathcal{H}$ of $G$, $G$ is $(\mathcal{H},b)$-colorable.  The \emph{fractional DP-chromatic number} is
        \begin{equation*}
            \chi_{_{DP}}^*(G) \,=\, \inf\{a/b \,:\, G\text{ is $(a,b)$-DP-colorable}\}.
        \end{equation*}
       When $\mathbb{P}(\text{$G$ is $(\mathcal{H}(G,a),b)$-colorable}) = p$, we say that $G$ is \emph{$(a,b)$-DP-colorable with probability $p$}. (Here $\mathcal{H}(G,a)$ is the random $a$-fold cover of $G$ defined in \S\ref{subsec:randomCover}.) Given $k \geq 1$, we define
       \[
        p^\ast(G,k) \,=\, \sup \{p \,:\, \exists a,\, b \in \N \text{ s.t.~$a/b \leq k$ and $G$ is $(a,b)$-DP-colorable with probability $p$}\}.
       \]
        We obtain two results for fractional-DP-coloring. The first is Proposition~\ref{pro:fractionalfirst}, which is a general result that implies Proposition~\ref{prop:first}.  

        \begin{pro}\label{pro:fractionalfirst}
        Let $\epsilon>0$ and let $G$ be a nonempty graph with 
            $\rho(G)\geq \exp(e/\epsilon)$. If $1\leq k\leq\rho(G)/\ln\rho(G)$, then $p^\ast(G,k) \leq \epsilon$.

		\end{pro}

        Proposition~\ref{pro:fractionalfirst} is a slight strengthening of \cite[Theorem 1.9]{BKZ} due to Kostochka, Zhu, and the first named author and is proved using the First Moment Method. We present the proof in \S\ref{subsec:FM}.
  
        Our second result is a version of Theorem~\ref{Degeneracy} for fractional DP-coloring:

        \begin{thm}\label{thm:fractionalDegeneracy}

            For all $\epsilon>0$, there is $d_0\in\N$ such that the following holds.  Let $G$ be a graph with degeneracy $d\geq d_0$ and let $k\geq(1+\epsilon)d/\ln d$.  Then $p^\ast(G,k) \geq 1-\epsilon$.
            
		\end{thm}

        Theorem~\ref{thm:fractionalDegeneracy} extends the result of Theorem~\ref{Degeneracy} to any graph whose degeneracy is high enough as a function of $\epsilon$ (regardless of how small it is when compared to the number of vertices in the graph), at the cost of replacing DP-coloring with fractional DP-coloring.  In particular, Proposition~\ref{pro:sparse} fails in the fractional setting. 
        The proof of Theorem~\ref{thm:fractionalDegeneracy} is analogous to that of Theorem~\ref{Degeneracy}  and will be presented in \S\ref{sec:fractional}. 

        \section{DP-colorability with probability close to $0$}\label{sec:nonKCol}

        \subsection{A First Moment argument}\label{subsec:FM}

        \noindent In this subsection we prove Proposition~\ref{pro:fractionalfirst}, which clearly implies Proposition~\ref{prop:first}. 

		\begin{obs}\label{subgraph_observation}
            If $G'$ is a subgraph of a graph $G$, then $p^\ast(G, k) \leq p^\ast(G',k)$ for all $k \geq 1$.
		\end{obs}
		\begin{proof}
            Given a cover $\mathcal{H} = (L,H)$ of $G$, the \emph{subcover of $\mathcal{H}$ corresponding to $G'$} is $\mathcal{H}'=(L',H')$ where $L'$ is the restriction of $L$ to $V(G')$ and $H'$ is the subgraph of $H$ with vertex set $V(H')=\bigcup_{u\in V(G')}L(u)$ that retains those and only those edges of $H$ that belong to the matchings corresponding to the edges of $G'$. 
        Clearly, $\mathcal{H}'$ is a cover of $G'$.
            
			Now suppose that $a$, $b\in\N$ satisfy $a/b\leq k$. 
			Note that the subcover $\mathcal{H'}$ of $\mathcal{H}(G,a)$ corresponding to $G'$ is a random  cover of $G'$ with the same distribution as $\mathcal{H}(G',a)$. If $\mathcal{H}'$ has no independent $b$-fold transversal, $\mathcal{H}$ also has no independent $b$-fold transversal. Thus,
			\begin{align*}
				&\mathbb{P}\left(\mathcal{H}(G,a)\text{ has an independent $b$-fold transversal}\right)\\
                \leq\,&\mathbb{P}\left(\mathcal{H}(G',a)\text{ has an independent $b$-fold transversal}\right) \,\leq\, p^\ast(G',k).
			\end{align*}
			Since this holds for all $a/b \leq k$, we conclude that $p^\ast(G,k) \leq p^\ast(G',k)$, as desired.
		\end{proof}

		We now prove Proposition~\ref{pro:fractionalfirst}, which we restate here for the reader's convenience.
 
		\begin{custompro}{\bf\ref{pro:fractionalfirst}}
            Let $\epsilon>0$ and let $G$ be a nonempty graph with 
            $\rho(G)\geq \exp(e/\epsilon)$. If $1\leq k\leq\rho(G)/\ln\rho(G)$, then $p^\ast(G,k) \leq \epsilon$.
		\end{custompro}
		\begin{proof}
			Suppose $a$, $b\in\N$ satisfy $1\leq a/b\leq k$. We need to argue that $G$ is $(a,b)$-DP-colorable with probability at most $\epsilon$. By Observation~\ref{subgraph_observation}, we may assume that $\rho(G)$ equals the density of $G$.  For ease of notation, let $n=|V(G)|$, $m=|E(G)|$, and $\rho=\rho(G)=m/n$.  

			We enumerate all $b$-fold transversals of $\mathcal{H}(G,a)=(L,H)$ as $\{I_i\}_{i=1}^s$ where $s={a\choose b}^n$.  Let $E_i$ be the event that $I_i$ is an independent $b$-fold transversal and let $X_i$ be the indicator random variable for the event $E_i$. Define $X=\sum_{i=1}^s X_i$.  Note that $X$ is the random variable that counts the number of independent $b$-fold transversals of $\mathcal{H}(G,a)$.  We have
			\begin{equation*}
				\mathbb{E}(X_i) \,=\,\left(\frac{{{a-b}\choose b}}{{a\choose b}}\right)^{m},\hspace{0.25in}\text{which implies}\hspace{0.25in}\mathbb{E}(X)\,=\,\sum_{i=1}^{s}\mathbb{E}\left(X_i\right)\,=\,{a\choose b}^{n}\left(\frac{{{a-b}\choose b}}{{a\choose b}}\right)^{m}.
			\end{equation*}
			Notice that, since $m > 0$, $a<2b$ implies $\mathbb{E}(X)=0$.  We will now establish an upper bound on $\mathbb{E}(X)$ when $a\geq2b$.  We see
			\begin{equation*}
				\mathbb{E}(X) \,=\,\exp\left(n\ln {a\choose b}+m\ln\left(\frac{{{a-b}\choose b}}{{a\choose b}}\right)\right) \,=\,\exp\left(n\left(\ln {a\choose b}+\rho\ln\left(\frac{{{a-b}\choose b}}{{a\choose b}}\right)\right)\right).
			\end{equation*}
			Now we write
			\begin{align*}
				\ln {a\choose b}+\rho\ln\left(\frac{{{a-b}\choose b}}{{a\choose b}}\right)\,&=\,\ln \left(\prod_{j=0}^{b-1}\frac{a-j}{b-j}\right)+\rho\ln\left(\prod_{j=0}^{b-1}\frac{a-b-j}{a-j}\right)\\
				&\leq\,\sum_{j=0}^{b-1}\left(\ln \left(\frac{a}{b-j}\right)+\rho\ln\left(1 - \frac{b}{a}\right)\right)\\
				&=\,\sum_{j=0}^{b-1}\left(\ln\left(\frac{a}{b}\right)+\ln b-\ln(b-j)+\rho\ln\left(1-\frac{b}{a}\right)\right)\\
				[\text{since $a/b \leq k$}]\qquad\qquad&\leq\,\sum_{j=0}^{b-1}\left(\ln k+\rho\ln\left(1-\frac{1}{k}\right)\right)+b\ln b-\ln b!\\
				[\text{since $\ln b! \geq b \ln b - b$}]\qquad\qquad&\leq\,\sum_{j=0}^{b-1}\left(1 + \ln k+\rho\ln\left(1-\frac{1}{k}\right)\right)\\
                [\text{since $\ln(1-x) \leq -x$ for $x < 1$}]\qquad\qquad&\leq\,\sum_{j=0}^{b-1}\left(1 + \ln k -  \frac{\rho}{k}\right)\\
                [\text{since $k \leq \rho/\ln \rho$}]\qquad\qquad&\leq\,\sum_{j=0}^{b-1}\left(1 + \ln \left(\frac{\rho}{\ln \rho}\right) -  \ln \rho\right)\\
                &=\, b(1 - \ln\ln \rho).
			\end{align*}
			It follows that
            \begin{equation}\label{Ineq:rho}
				\mathbb{E}(X) \,\leq\, \exp\left(n b(1 - \ln \ln \rho)\right) \,\leq\, \exp\left(1 - \ln \ln \rho\right) \,\leq\, \epsilon,
			\end{equation}
            since $\rho \geq \exp(e/\epsilon)$ by assumption. Therefore, by Markov's Inequality we have
            \begin{equation*}
                \mathbb{P}(X=0) \,=\, 1-\mathbb{P}(X\geq1) \,\geq\, 1-\mathbb{E}(X) \,\geq\ 1-\epsilon. \qedhere
            \end{equation*}
		\end{proof}

    \subsection{Sparse graphs with low probability of $k$-DP-colorability}\label{sec:sparse}

    \noindent In this subsection we construct graphs $G$ with polylogarithmic density that have low probability of $k$-DP-colorability for any $k \leq 2\rho(G)$.

		\begin{custompro}{\bf\ref{pro:sparse}}
            For any $\epsilon > 0$ and $n_0\in\N$, there is a graph $G$ with $n \geq n_0$ vertices such that $\rho(G) \geq \left( \ln n / \ln\ln n\right)^{1/3}$ but, for every $k\leq 2\rho(G)$, $G$ is $k$-DP-colorable with probability less than $\epsilon$.
		\end{custompro}
        \begin{proof}
            Fix $\epsilon > 0$ and let $q \in \mathbb{N}$ be sufficiently large as a function of $\epsilon$. Consider the graph $G=tK_q$, the disjoint union of $t$ copies of $K_q$, where
            \[
                t \,=\, \ln(1/\epsilon) \, (q-1)!^{{q}\choose2}.
            \]
            Let us denote the copies of $K_q$ in $G$ by $K^1$, \ldots, $K^t$. Note that $\rho = \rho(G) = (q-1)/2$. Suppose that $k \leq 2\rho = q-1$. For each $r \in [t]$, let $A_r$ be the event that in the cover $\mathcal{H}(G,k)$, the vertices $(u,i)$ and $(v,i)$ are adjacent for all $uv \in E(K^r)$ and $i \in [k]$. Note that if $A_r$ happens for some $r \in [t]$, then $G$ is not $\mathcal{H}(G,k)$-colorable (because it is impossible to color $K^r$). The events $A_1$, \ldots, $A_t$ are mutually independent, so
            \begin{align*}
                \mathbb{P}(\text{$G$ is $\mathcal{H}(G,k)$-colorable}) \,\leq\, \mathbb{P}\left(\bigcap_{r \in [t]} \overline{A_r}\right) \,&=\, \left(1 - (k!)^{-{q \choose 2}}\right)^t \,<\, \exp\left(- t (k!)^{-{q \choose 2}}\right) \,\leq\, \epsilon.
            \end{align*}
            It remains to observe, using the bound $(q-1)! \leq q^q$ and assuming $q > \ln(1/\epsilon)$, that 
            \[
                |V(G)|\,=\, tq \,=\, \ln(1/\epsilon) \, (q-1)!^{{q \choose 2}}\,q \,\leq\, (q-1)!^{{q \choose 2}} \, q^2 \, \leq\, q^{q^3/2},
            \]
            and we have, for large enough $q$,
            \[
                \left(\frac{\ln q^{q^3/2}}{\ln \ln q^{q^3/2}}\right)^{1/3} \,=\, \left(\frac{q^3\ln q}{2(3\ln q - \ln 2 + \ln \ln q)}\right)^{1/3} \,\leq\, \left(\frac{q^3}{6}\right)^{1/3} \,=\, \frac{q}{6^{1/3}} \,<\, \frac{q-1}{2} \,=\, \rho. \qedhere
            \]
            \end{proof}

        \section{DP-colorability with probability close to $1$ for dense graphs}\label{sec:kCol}

		\noindent Now we prove Theorem~\ref{thm:slidingsharpness}, which we state here again for convenience: 

		\begin{customthm}{\bf\ref{thm:slidingsharpness}}
            For all $\epsilon > 0$ and $s \in [0, 1/3)$, there is $n_0 \in \N$ such that the following holds. Suppose $G$ is a graph with $n \geq n_0$ vertices such that $\rho(G)\geq n^{1-s}$, and
            \[
                k \,\geq\, \left(1+\epsilon\right)\left(1+\frac{s}{1-2s}\right)\frac{\rho(G)}{\ln \rho(G)}.
            \]
            Then $G$ is $k$-DP-colorable with probability at least $1-\epsilon$.
		\end{customthm}

		\begin{proof} 
			Without loss of generality, we assume that $\epsilon < 1$. Consider an arbitrary graph $G$ with $n$ vertices, $m$ edges, and maximum density $\rho=\rho(G)$ such that $\rho\geq n^{1-s}$, and let $k$ satisfy the bound in the statement of the theorem.  Throughout the proof, we shall treat $\epsilon$ and $s$ as fixed, while $n$ will be assumed to be sufficiently large as a function of $\epsilon$ and $s$. In particular, when we employ asymptotic notation, such as $O(\cdot)$, $o(\cdot)$, and $\Omega(\cdot)$, it will always be with respect to $n \to \infty$, while the implied constants will be allowed to depend on $\epsilon$ and $s$. As usual, the asymptotic notation $\tilde{O}(\cdot)$ and $\tilde{\Omega}(\cdot)$ hides polylogarithmic factors.
   
            We enumerate all transversals of $\mathcal{H}(G,k)=(L,H)$ as $\{I_i\}_{i=1}^{k^n}$. Throughout the proof, the variables $i$ and $j$ will range over $[k^n]$. Let $E_i$ be the event that $I_i$ is independent and let $X_i$ be the indicator random variable for the event $E_i$. Define $X=\sum_{i=1}^{k^n} X_i$, so $X$ is the random variable that counts the $\mathcal{H}(G,k)$-colorings of $G$. For convenience, we set $p=(k-1)/k$. Clearly, \[\E(X)\,=\,k^n\left(\frac{k-1}{k}\right)^m \,=\, k^n p^m.\]
			Notice that
			\begin{align}
				\Var(X) \,=\,\sum_{(i,j)}\Cov(X_i,X_j) \,=\, \sum_{(i,j)}\left(\E(X_iX_j)-\E(X_i)\E(X_j)\right) \,=\, \sum_{(i,j)}\E(X_iX_j)  \,-\, \E(X)^2. \label{eq:variance}
			\end{align}
			Now, for each $(i,j)\in[k^n]^2$, we need to compute $\E(X_iX_j)$. To this end, for every edge $uv\in E(G)$, let $I_i\cap L(u) = \{(u,i_u)\}$, $I_i\cap L(v) = \{(v,i_v)\}$, $I_j\cap L(u) = \{(u,j_u)\}$, and $I_j\cap L(v) = \{(v,j_v)\}$. 
            Define $A_{uv}$ as the event $(u,i_u)(v,i_v)\not\in E(H)$ and $B_{uv}$ as the event $(u,j_u)(v,j_v)\not\in E(H)$, and let $E_{uv}=A_{uv}\cap B_{uv}$. 
            Now we consider three cases.
            \begin{enumerate}[label=(\alph*)]
                \item\label{case:a} $i_u= j_u$ and $i_v= j_v$.
            \end{enumerate}
            In this case, we have $\mathbb{P}(E_{uv})= (k-1)/k = p$.
    \begin{enumerate}[label=(\alph*),resume]
                \item\label{case:b} $i_u\neq j_u$ and $i_v= j_v$, or $i_u= j_u$ and $i_v\neq j_v$.
            \end{enumerate}
			In this case, regardless of whether $i_v = j_v$ or $i_u = j_u$, we have
			\begin{equation*}
				\mathbb{P}(E_{uv}) \,=\, 
                \frac{k-2}{k}\,=\,p^2-\frac{1}{k^2} \,\leq\, p^2.
			\end{equation*}
    \begin{enumerate}[label=(\alph*),resume]
               \item\label{case:c} $i_u\neq j_u$ and $i_v\neq j_v$.
            \end{enumerate}
			To compute $\mathbb{P}(E_{uv})$ in this case, we let $C_{uv}$ be the event $(u,i_u)(v,j_v)\in E(H)$.  Note that $C_{uv}\subseteq B_{uv}$ since $E_H(L(u),L(v))$ is a matching.  
            Therefore, 
			\begin{align*}
				\mathbb{P}(E_{uv})\,&=\,\mathbb{P}(A_{uv}) \cdot\mathbb{P}(B_{uv}|A_{uv}) \\
                &=\,\mathbb{P}(A_{uv}) \cdot\left(\mathbb{P}\left( C_{uv}|A_{uv}\right) +\mathbb{P}\left(B_{uv} \cap \overline{C_{uv}} |A_{uv}\right)\right) \\
				&=\,\mathbb{P}(A_{uv})\cdot\left(\mathbb{P}\left( C_{uv}|A_{uv}\right) +\mathbb{P}\left(\overline{C_{uv}}|A_{uv}\right)\cdot\mathbb{P}\left(B_{uv}|A_{uv}\cap \overline{C_{uv}}\right)\right)\\
				&=\,\left(\frac{k-1}{k}\right) \left(\frac{1}{k-1}+\left(\frac{k-2}{k-1}\right)^2\right)\,=\,p^2+\frac{1}{k^2(k-1)}.
			\end{align*}
   
			Let the number of edges in $E(G)$ satisfying the conditions from cases \ref{case:a}, \ref{case:b}, and \ref{case:c} be $\alpha_{i,j}$, $\beta_{i,j}$, and $\gamma_{i,j}$, respectively. Since the matchings in $H$ corresponding to different edges of $G$ are chosen independently, we conclude that
			\begin{equation*}
				\E(X_iX_j) \,\leq\,p^{\alpha_{i,j}}p^{2\beta_{i,j}}\left(p^2+\frac{1}{k^2(k-1)}\right)^{\gamma_{i,j}}.
			\end{equation*}
			Using \eqref{eq:variance} and the fact that $\alpha_{i,j}+\beta_{i,j}+\gamma_{i,j}=m$, we now obtain
			\begin{align*}
				\frac{\Var(X)}{\E(X)^2}\,&=\, \frac{\sum_{(i,j)}\E(X_iX_j)}{k^{2n}p^{2m}} - 1\\
				&\leq\,k^{-2n}p^{-2m}\sum_{(i,j)}p^{\alpha_{i,j}}p^{2\beta_{i,j}}\left(p^2+\frac{1}{k^2(k-1)}\right)^{\gamma_{i,j}} - 1\\
                &=\,k^{-2n}\sum_{(i,j)}p^{-\alpha_{i,j}}\left(\frac{p^2+\frac{1}{k^2(k-1)}}{p^2}\right)^{\gamma_{i,j}} - 1\\
				&=\,k^{-2n}\sum_{(i,j)}\left(1+\frac{1}{k-1}\right)^{\alpha_{i,j}}\left(1+\frac{1}{(k-1)^3}\right)^{\gamma_{i,j}} - 1.
			\end{align*}
            Note that, for given $i$ and $j$, $\alpha_{i,j}$ is exactly the number of edges of $G$ with both endpoints in the set $\{v\in V(G):L(v)\cap I_i= L(v)\cap I_j\}$. Since the density of every subgraph of $G$ is at most $\rho$, we conclude that if $|I_i \cap I_j| = \nu$, then $\alpha_{i,j} \leq \mu(\nu)$, where
            \[
                \mu(\nu) \,=\, \min\left\{{\nu \choose 2},\, \rho \nu\right\} \,\leq\, \min \{\nu^2, \, \rho \nu\}.
            \]
            Also, $\gamma_{i,j} \leq m \leq \rho n$. As there are ${n\choose \nu}k^n(k-1)^{n-\nu}$ pairs $(i,j)$ with $|I_i\cap I_j|=\nu$, we have
			\begin{align*}
				\frac{\Var(X)}{\E(X)^2}
				\,
                &\leq\,k^{-2n}\sum_{\nu=0}^n\left[{n\choose \nu}k^n(k-1)^{n-\nu}\left(1+\frac{1}{k-1}\right)^{\mu(\nu)}\left(1+\frac{1}{(k-1)^3}\right)^{\rho n}\right] - 1\\
                &=\, \left(1+\frac{1}{(k-1)^3}\right)^{\rho n}\sum_{\nu=0}^n\left[{ n\choose \nu}{k^{-\nu}}\left(1 - \frac{1}{k}\right)^{n - \nu}\left(1+\frac{1}{k-1}\right)^{\mu(\nu)}\right] \,-\, 1.
			\end{align*}
            For ease of notation, we define
			\begin{equation*}
				g(\nu)\,=\,{ n\choose \nu}{k^{-\nu}}\left(1 - \frac{1}{k}\right)^{n - \nu}\left(1+\frac{1}{k-1}\right)^{\mu(\nu)}.
			\end{equation*}
			Note that $g(\nu) = 0$ for $\nu \notin [0:n]$. Observe that, since $\rho \geq n^{2/3}$ and $k = \tilde{\Omega}(\rho)$,
            \[
                \left(1+\frac{1}{(k-1)^3}\right)^{\rho n} \,\leq\, 
                \exp\left(\frac{\rho n}{(k-1)^3}\right) \,=\, 1 + \tilde{O}\left(\frac{n}{\rho^2}\right) \,=\, 1 + \tilde{O}(n^{-1/3}) \,=\, 1 + o(1).
            \]
            (Recall that the asymptotic notation here is with respect to $n \to \infty$.) Therefore,
            \begin{equation}\label{eq:VarBound}
                \frac{\Var(X)}{\E(X)^2}
				\,\leq\,(1 + o(1))\sum_{\nu=0}^ng(\nu) \,-\, 1.
            \end{equation}
            The key to our analysis is to divide the summation according to the magnitude of $\nu$. Define 
            \[
                \delta \,=\, \frac{1 - 3s}{4},
            \]
            and note that $\delta > 0$ since $s < 1/3$.
            
			\begin{cl}\label{claimIneq} The following asymptotic bounds hold as $n \to \infty$:
				\begin{equation*}
					\sum_{\nu \,\leq\, n^{s + \delta}}g(\nu) \,\leq\, 1 + o(1), \qquad \sum_{n^{s + \delta} \,<\, \nu \,\leq\, k - 1} g(\nu) \,=\, o(1), \qquad \text{and} \qquad \sum_{\nu = k}^{n} g(\nu)\,=\, o(1).
				\end{equation*}
			\end{cl}
   Claim~\ref{claimIneq} and inequality \eqref{eq:VarBound} imply that $\Var(X)/\mathbb{E}(X)^2=o(1)$, so, by Chebyshev's inequality, 
            \begin{equation*}
                \mathbb{P}(X>0) \,=\,1-\mathbb{P}(X=0) \,\geq\, 1-\frac{\Var(X)}{\mathbb{E}(X)^2} \,\geq\, 1-o(1),
            \end{equation*}
            and hence this probability exceeds $1- \epsilon$ for all large enough $n$. 
            Consequently, our proof of Theorem~\ref{thm:slidingsharpness} will be complete once we verify Claim~\ref{claimIneq}.

            To prove the first bound in Claim~\ref{claimIneq}, we note that for $\nu \leq n^{s+\delta}$, we have $\mu(\nu) \leq \nu^2 \leq \nu n^{s+\delta}$. Hence, for $\nu \leq n^{s+\delta}$,
            \begin{align*}
               g(\nu) \,\leq\, { n\choose \nu}{k^{-\nu}}\left(1 - \frac{1}{k}\right)^{n - \nu}\left(1+\frac{1}{k-1}\right)^{\nu n^{s+\delta}}.
            \end{align*}
            Using the bounds $1 + x \leq e^x\leq 1 + 2x$ valid for all $x \in (0,1)$, we write
            \[
                \left(1+\frac{1}{k-1}\right)^{n^{s+\delta}} \,\leq\, \exp\left(\frac{n^{s+\delta}}{k-1}\right) \,\leq\, \exp\left(n^{-1 + 2s  + 2\delta}\right) \,\leq\, 1 + 2n^{-1 + 2s + 2\delta},
            \]
            where the second inequality holds for all large enough $n$ since $k = \tilde
{\Omega}(n^{1-s})$. Thus,
            \[
                g(\nu) \,\leq\, {n \choose \nu} \left(\frac{1 + 2n^{-1 + 2s + 2\delta}}{k}\right)^{\nu} \left(1 - \frac{1}{k}\right)^{n - \nu}.
            \]
            By the binomial formula,
            \[
                \sum_{\nu \,\leq\, n^{s + \delta}} g(\nu) \,\leq\, \sum_{\nu = 0}^n  {n \choose \nu} \left(\frac{1 + 2n^{-1 + 2s + 2\delta}}{k}\right)^{\nu} \left(1 - \frac{1}{k}\right)^{n - \nu} \,=\, \left(1 + \frac{2n^{-1 + 2s + 2\delta}}{k}\right)^n.
            \]
            Using again that $k = \tilde{\Omega}(n^{1-s})$ and assuming that $n$ is large enough, we obtain
            \[
                \left(1 + \frac{2n^{-1 + 2s + 2\delta}}{k}\right)^n \,\leq\, \left(1 + n^{-2 + 3s + 3\delta}\right)^n \,\leq\, 1 + n^{-1 + 3s + 3\delta} \,=\, 1 + o(1),
            \]
            where in the last step we use that $3\delta < 1 - 3s$.

            For the second part of Claim~\ref{claimIneq}, note that $\mu(\nu) \leq \nu (k-1)$ for $\nu \leq k - 1$. Since $(1+1/x)^x < e$ for all $x > 0$ and ${n \choose \nu} \leq (en/\nu)^\nu$ for all $\nu \in [0 : n]$, we see that for $n^{s+\delta}<\nu \leq k-1$,
            \begin{align*}
                g(\nu) \,\leq\, {n\choose \nu}{k^{-\nu}}\left(1+\frac{1}{k-1}\right)^{\nu(k-1)} \,\leq\, \left(\frac{e^2 n}{\nu k}\right)^\nu \,\leq\, \left(\frac{e^2 n^{1 - s - \delta}}{k}\right)^\nu \,\leq\, n^{-\delta \nu/2} \,\leq\, n^{-\delta\, n^{s+\delta}/2},
            \end{align*}
            where the second to last inequality holds for all large enough $n$ since $k = \tilde{\Omega}(n^{1-s})$. Therefore,
            \[
                \sum_{n^{s + \delta} \,<\, \nu \,\leq\, k - 1} g(\nu) \,\leq\, n \cdot n^{-\delta\, n^{s+\delta}/2} \,=\, n^{-\delta\, n^{s+\delta}/2 + 1} \,=\, o(1).
            \]

            Finally, we consider the case $\nu \geq k$. This is the only part of the proof where the parameter $\rho$ is used and where the specific constant factor in front of $\rho/\ln \rho$ in the lower bound on $k$ plays a role. Given $\nu \geq k$, we use the bound $\mu(\nu) \leq \rho \nu$ to write
            \[
                g(\nu) \,\leq\, {n\choose \nu}{k^{-\nu}}\left(1+\frac{1}{k-1}\right)^{\rho \nu}.
            \]
            Since $1/(1-\epsilon/2) < 1 + \epsilon$ for $0 < \epsilon < 1$, for all large enough $n$, we have
            \[
                k - 1 \,\geq\, (1 + \epsilon) \frac{(1-s)\rho}{(1-2s)\ln \rho} - 1 \,\geq\, \frac{(1-s)\rho}{(1 - \epsilon/2)(1-2s)\ln \rho}.
            \]
            Since $1 + x \leq e^x$ for all $x  \in \mathbb{R}$, we conclude that
            \[
            \left(1+\frac{1}{k-1}\right)^{\rho} \,\leq\, \exp\left(\frac{\rho}{k-1}\right) \,\leq\, \exp\left(\frac{(1 - \epsilon/2)(1-2s) \ln \rho}{1-s}\right)\,=\, \rho^{\frac{(1 - \epsilon/2)(1-2s)}{1-s}}.
            \]
            Therefore,
            \[
                \frac{1}{k}\, \rho^{\frac{(1 - \epsilon/2)(1-2s)}{1-s}} \,\leq\, \frac{(1 - 2s) \ln \rho}{(1+\epsilon)(1 - s) \rho}\, \rho^{\frac{(1 - \epsilon/2)(1-2s)}{1-s}}  
                \,=\, \frac{(1 - 2s) \ln \rho}{(1+\epsilon)(1 - s)} \, \rho^{-\frac{s}{1-s} - \frac{\epsilon(1-2s) }{2(1-s)}} \,\leq\, n^{-s - \epsilon(1-2s)/4},
            \]
            where the last inequality holds for large $n$ since $\rho \geq n^{1-s}$. Thus, for each $\nu \geq k$,
            \[
                g(\nu) \,\leq\, \left(\frac{en}{\nu} \cdot n^{-s - \epsilon(1-2s)/4}\right)^\nu \,\leq\, n^{-\epsilon(1-2s)\nu/8},
            \]
            as $k = \tilde{\Omega}(n^{1-s})$. Hence,
            \[
                \sum_{\nu = k}^n g(\nu) \,\leq\, n \cdot n^{-\epsilon(1-2s)\nu/8} \,=\, n^{-\epsilon(1-2s)\nu/8 + 1} \,=\, o(1),
            \]
            and the proof is complete.
		\end{proof}

    \section{DP-colorability based on degeneracy}\label{sec:degeneracy}

        \subsection{Greedy Transversal Procedure}\label{subsec:prelimLem}

		\noindent Consider a graph $G$ and an ordering of its vertices $(v_i)_{i\in[n]}$.  Let $d_i^-$ be \emph{the back degree of $v_i$}, which is equal to the number of neighbors of $v_i$ whose indices are less than $i$.  Recall that a graph $G$ is {$d$-degenerate} if there exists some ordering of vertices in $V(G)$ such that $d_i^-\leq d$ for all $i\in[n]$, and the {degeneracy} of a graph $G$ is the smallest $d$ such that $G$ is $d$-degenerate.

        In order to prove Theorems~\ref{Degeneracy} and \ref{thm:fractionalDegeneracy}, we shall employ the following greedy procedure to select a $b$-fold transversal from a given cover:

		\begin{proc}[Greedy Transversal Procedure (GT Procedure)]\label{proc:GT}

			This procedure takes as input a $d$-degenerate graph $G$, an $a$-fold cover $\mathcal{H}=(L,H)$ of $G$ with $L(v) = \{(v,j) : j \in [a]\}$ for all $v \in V(G)$, and $b \in \N$ satisfying $b \leq a$. 
            The procedure gives as output a $b$-fold transversal $T$ for $\mathcal{H}$. For ease of notation we let $n=|V(G)|$. For $(v,j) \in L(v)$, we call $j$ the \emph{index} of $(v,j)$.
			\begin{enumerate}

				\item Order the vertices $V(G)$ as $(v_i)_{i\in[n]}$ so that $d_i^-\leq d$ for all $i\in[n]$.

				\item Initialize $T=\emptyset$ and repeat the following as $i$ goes from $1$ to $n$.
				\begin{enumerate}

					\item Let $L'(v_i)=L(v_i)\setminus N_H(T)$. 

                    \item If $|L'(v_i)|\geq b$, form a $b$-element subset $T(v_i) \subseteq L'(v_i)$ by selecting $b$ elements of $L'(v_i)$ with the smallest indices; otherwise, let $T(v_i) = \{(v_i,j) : j \in [b]\}$.



					\item Replace $T$ with $T\cup T(v_i)$.

				\end{enumerate}

				\item Output $T$.

			\end{enumerate}

		\end{proc}
        By construction, the output of this procedure is a $b$-fold transversal for $\mathcal{H}$. Furthermore, this transversal is independent if and only if $|L'(v_i)| \geq b$ for all $i \in [n]$.\footnote{Alternatively, we could say that the procedure fails if $|L'(v_i)| < b$ for some $i$. However, for the ensuing analysis it will be more convenient to assume that the procedure runs for exactly $n$ steps and that the sets $L'(v_i)$ and $T(v_i)$ are well defined for all $i$.} 

        A crucial role in our analysis 
        is played by the notion of negative correlation. A collection $(Y_i)_{i\in [k]}$ of $\{0,1\}$-valued random variables is \emph{negatively correlated} if for every subset $I\subseteq[k]$, we have
		\begin{equation*}
			\mathbb{P}\left(\bigcap_{i\in I}\{Y_i = 1\}\right) \,\leq\,\prod_{i\in I}\mathbb{P}(Y_i = 1).
		\end{equation*}
        This notion is made useful by the fact that sums of negatively correlated random variables satisfy Chernoff--Hoeffding style bounds, as discovered by Panconesi and Srinivasan~\cite{PS}. The exact formulation we use is from the paper \cite{Molloy} by Molloy: 

		\begin{lem}[{\cite[Theorem 3.2]{PS}, \cite[Lemma 3]{Molloy}}]\label{Bernshteyn1}
			Let $(X_i)_{i\in[k]}$ be $\{0,1\}$-valued random variables.  Set $Y_i=1-X_i$ and $X=\sum_{i\in[k]}X_i$.  If $(Y_i)_{i\in [k]}$ are negatively correlated, then
			\begin{equation*}
				\mathbb{P}\left(X<\mathbb{E}(X)-t\right) \,<\,\exp\left(-\frac{t^2}{2\mathbb{E}(X)}\right) \quad \text{for all } 0 < t \leq \mathbb{E}(X).
			\end{equation*}
		\end{lem} 

		Now consider running the GT Procedure on a $d$-degenerate graph $G$ and the random $a$-fold cover $\mathcal{H}(G,a) = (L,H)$, producing a $b$-fold transversal $T$. 
        For each $i\in[n]$ and $j\in[a]$, let $X_{i,j}$ be the indicator random variable of the event $\{(v_i,j) \in L'(v_i)\}$ and let $Y_{i,j} = 1 - X_{i,j}$. 
        %

		\begin{lem}\label{lem:negCor}
			Consider the set of random variables $X_{i,j}$ as defined above.
            \begin{enumerate}[label={\normalfont(\roman*)}]
                \item\label{item:expectation} For all $i \in [n]$ and $j \in [a]$, we have
                $
                 \displaystyle   \mathbb{E}(X_{i,j}) \,\geq\, \left(1-\frac{b}{a}\right)^d.
                $
                \item\label{item:correlation} For each $i \in [n]$, the variables $(Y_{i,j})_{j\in[a]}$ are negatively correlated.
            \end{enumerate}
		\end{lem}
		\begin{proof}
            Similar statements have appeared in the literature; see, e.g., \cites[272]{Molloy}[Lemma~4.3]{BKZ}[Lemma 3.2]{B2}. Since our setting is somewhat different from these references, we include the proof here for completeness.
            
            Fix any $i \in [n]$. Write $t = d^-_i$ and let $u_1$, \ldots, $u_t$ be the neighbors of $v_i$ that precede it in the ordering $(v_i)_{i\in[n]}$. The sets $T(u_1)$, \ldots, $T(u_t)$ are constructed in a way that is probabilistically independent from the matchings in $H$ corresponding to the edges $u_1v_i$, \ldots, $u_tv_i$. Therefore, 
            \[
                S_r \,=\, \{j \in [a] \,:\, (v_i, j) \in N_H(T(u_r))\}
            \]
            for $r \in [t]$ are mutually independent uniformly random $b$-element subsets of $[a]$. 
            Hence,
		\begin{equation*}
			\mathbb{E}(X_{i,j}) \,=\, \mathbb{P}(j \notin S_r \text{ for all } r \in [t]) \,=\, \left(1-\frac{b}{a}\right)^{t} \,\geq\, \left(1-\frac{b}{a}\right)^d,
		\end{equation*}
    where in the last step we use that $t = d^-_i \leq d$. This proves part \ref{item:expectation}.

            For part \ref{item:correlation}, we first note that if $b = a$, then the output of the GT Procedure on $\mathcal{H}(G,a)$ is deterministic, 
            so the negative correlation holds trivially in this case. Now suppose that $b < a$ and consider any $I \subset [a]$ and $j \in [a] \setminus I$. We claim that
            \begin{equation}\label{eq:cond_prob_bound1}
                \mathbb{P}\left( I \subseteq S \,\middle\vert\, j \notin S\right) \,\geq\, \mathbb{P}\left( I \subseteq S \right),
            \end{equation}
            where $S = \bigcup_{r \in [t]} S_r$. 
            To show \eqref{eq:cond_prob_bound1}, pick $b$-element subsets $S_1'$, \ldots, $S_t'$ of $[a] \setminus \{j\}$ independently and uniformly at random and let $S' = \bigcup_{r \in [t]} S_r'$. The event $\{j \notin S\}$ occurs if and only if $j \notin S_r$ for all $r \in [t]$. 
            It follows that the distribution of the tuple $(S_1, \ldots, S_t)$ conditioned on the event $\{j \notin S\}$ is the same as the distribution of $(S_1',\ldots,S_t')$. Therefore, we need to argue that
            \[
                \mathbb{P}\left(I \subseteq S'\right) \,\geq\, \mathbb{P}\left(I \subseteq S\right).
            \]
            To this end, we employ a coupling argument. Form $b$-element subsets $S_r'' \subseteq [a] \setminus \{j\}$ as follows: If $j \notin S_r$, then let $S_r'' = S_r$; otherwise, pick $j_r \in [a] \setminus S_r$ uniformly at random and let $S_r'' = (S_r \setminus \{j\}) \cup \{j_r\}$. By construction, $S_1''$, \ldots, $S_t''$ are independent and uniformly random $b$-element subsets of $[a] \setminus \{j\}$, so, letting $S'' = \bigcup_{r \in [t]}S''_r$, we can write
            \begin{align*}
                \mathbb{P}\left(I \subseteq S'\right) \,=\, \mathbb{P}\left(I \subseteq S''\right) \,\geq\, \mathbb{P}\left(I \subseteq S\right),
            \end{align*}
            where the inequality is satisfied because $S'' \supseteq S \setminus \{j\}$. 

            Since for any events $A$, $B$, the inequalities $\mathbb{P}(A | B) \geq \mathbb{P}(A)$ and $\mathbb{P}(B | A) \geq \mathbb{P}(B)$ are equivalent, we conclude from \eqref{eq:cond_prob_bound1} that $\mathbb{P}( j \notin S | I \subseteq S) \geq \mathbb{P}(j \notin S)$, or, equivalently, $\mathbb{P}( j \in S | I \subseteq S) \leq \mathbb{P}(j \in S)$. Using the definition of the set $S$, this inequality can be rewritten as
            \begin{equation}\label{eq:step}
                \mathbb{P}\left( Y_{i,j} = 1 \,\middle\vert\, \bigcap_{j' \in I} \{Y_{i,j'} = 1\}\right) \,\leq\, \mathbb{P}\left( Y_{i,j} = 1 \right).
            \end{equation}
            
            Finally, given any set $I = \{j_1, \ldots, j_k\} \subseteq [a]$, we apply \eqref{eq:step} iteratively with $\{j_1, \ldots, j_{\ell - 1}\}$ in place of $I$ and $j_\ell$ in place of $j$ to obtain the desired inequality
            \[
                \mathbb{P}\left(\bigcap_{j\in I}\{Y_{i,j} = 1\}\right) \,\leq\,\prod_{j\in I}\mathbb{P}(Y_{i,j} = 1). \qedhere
            \]
		\end{proof}

	\subsection{DP-colorability based on degeneracy}\label{sec:upperDegeneracy}

		\noindent We now apply the GT Procedure to prove Theorem~\ref{Degeneracy}, whose statement we repeat here for convenience.

		\begin{customthm}{\bf\ref{Degeneracy}}
			For all $\epsilon \in (0,1/2)$, there is $n_0\in\N$ such that the following holds.  Let $G$ be a graph with $n\geq n_0$ vertices and degeneracy $d$ such that $d\geq \ln^{2/\epsilon}n$ and let $k\geq(1+\epsilon)d / \ln d$.  Then $G$ is $k$-DP-colorable with probability at least $1-\epsilon$.
		\end{customthm}
		\begin{proof}
			We construct a transversal $T$ using the GT Procedure with inputs $G$, $a = k$, the cover $\mathcal{H}(G,k)$, and $b=1$. Our aim is to argue that, with probability at least $1 - \epsilon$, the sets $L'(v_i)$ for all $i \in [n]$ are nonempty, and hence $T$ is independent with probability at least $1 - \epsilon$, as desired.
   

			Recall from \S\ref{subsec:prelimLem} that for $i \in [n]$ and $j \in [k]$, $X_{i,j}$ is the indicator random variable of the event $\{(v_i,j) \in L'(v_i)\}$. 
            Let
			\begin{equation*}
				X_i \,=\, \sum_{j\in[k]}X_{i,j} \,=\, |L'(v_i)|.
			\end{equation*}
            Using  Lemma~\ref{lem:negCor}\ref{item:expectation} and the linearity of expectation, we get
            \[
                \mathbb{E}(X_i) \,\geq\, k \left(1 - \frac{1}{k}\right)^d \,\geq\, k \, \exp\left(-(1 - o(1))\frac{d}{k}\right) \,>\, d^{2\epsilon/3},
            \]
            where the last inequality holds for $\epsilon < 1/2$ assuming that $n$ is large enough as a function of $\epsilon$ (since then $d$ and $k$ are also large). 
            Combining Lemma~\ref{lem:negCor}\ref{item:correlation} with Lemma~\ref{Bernshteyn1}, we obtain
            \begin{align*}
                \mathbb{P}(X_i = 0) \,\leq\, \mathbb{P}\left(X_i < \frac{\mathbb{E}(X_i)}{2}\right) \,<\, \exp\left(-\frac{\mathbb{E}(X_i)}{8}\right) \,\leq\, e^{-d^{2\epsilon/3}/8}  \,\leq\, e^{-\ln^{4/3}n/8}\,<\, \frac{\epsilon}{n},
            \end{align*}
            where the last inequality holds for large $n$. By the union bound, it follows that
            \[
                \mathbb{P}(X_i = 0 \text{ for some } i \in [n]) \,<\, \epsilon,
            \]
            and we are done.
		 \end{proof}
 


	\subsection{Fractional DP-colorability based on degeneracy}\label{sec:fractional}

        \noindent Here we prove Theorem~\ref{thm:fractionalDegeneracy}:

        \begin{customthm}{\bf\ref{thm:fractionalDegeneracy}}
			For all $\epsilon>0$, there is $d_0\in\N$ such that the following holds.  Let $G$ be a graph with degeneracy $d\geq d_0$ and let $k\geq(1+\epsilon)d/\ln d$.  Then $p^\ast(G,k) \geq 1-\epsilon$.
		\end{customthm}
    \begin{proof}
            Without loss of generality, we may assume that $\epsilon < 1$. Take $a$, $b \in \N$ such that
            \[
                (1 + \epsilon/2)\frac{d}{\ln d} \,\leq\, \frac{a}{b} \,\leq\, k.
            \]
            Here we may---and will---assume that $a$ and $b$ are sufficiently large as functions of $n$. We construct a $b$-fold transversal $T$ using the GT Procedure with inputs $G$ and $\mathcal{H}(G,a)$. Our aim is to show that the sets $L'(v_1)$, \ldots, $L'(v_n)$ all have size at least $b$ with probability at least $1 - \epsilon$, which implies that $T$ is independent with probability at least $1 - \epsilon$, as desired.

            As in the proof of Theorem~\ref{Degeneracy} presented in \S\ref{sec:upperDegeneracy}, we recall that for $i \in [n]$ and $j \in [a]$, $X_{i,j}$ is the indicator random variable of the event $\{(v_i,j) \in L'(v_i)\}$, and define
			\begin{equation*}
				X_i \,=\, \sum_{j\in[a]}X_{i,j} \,=\, |L'(v_i)|.
			\end{equation*}
            Using  Lemma~\ref{lem:negCor}\ref{item:expectation} and the linearity of expectation, we get
            \begin{align*}
                \mathbb{E}(X_i) \,\geq\, a \left(1 - \frac{b}{a}\right)^d \,&\geq\, b \cdot (1 + \epsilon/2)\frac{d}{\ln d} \cdot \left(1 - \frac{\ln d}{(1 + \epsilon/2)d}\right)^d \\
                &\geq\,b \cdot (1 + \epsilon/2)\frac{d}{\ln d} \cdot d^{-1/(1+\epsilon/2)} \, >\, b\, d^{\epsilon/3},
            \end{align*}
            where in the last step we use that $\epsilon < 1$ and $d$ is large as a function of $\epsilon$. In particular, assuming $d$ is large enough, we have $\mathbb{E}(X_i) > 2b$. By Lemma~\ref{lem:negCor}\ref{item:correlation} and Lemma~\ref{Bernshteyn1}, we have
            \begin{align*}
                \mathbb{P}(X_i < b) \,\leq\, \mathbb{P}\left(X_i < \frac{\mathbb{E}(X_i)}{2}\right) \,<\, \exp\left(-\frac{\mathbb{E}(X_i)}{8}\right) \,\leq\, e^{-b/4}  \,<\, \frac{\epsilon}{n},
            \end{align*}
            where for the last inequality we assume that $b$ is large enough as a function of $n$. By the union bound, it follows that
            \[
                \mathbb{P}(X_i < b \text{ for some } i \in [n]) \,<\, \epsilon,
            \]
            and the proof is complete.
    \end{proof}

	\section{DP-thresholds}\label{subsection:Threshold}
 
		\noindent Combining Proposition~\ref{prop:first} and Theorem~\ref{thm:slidingsharpness}, we obtain DP-thresholds and sharp DP-thresholds for sufficiently dense graph sequences:

		\begin{thm}\label{thm:threshold}
			Let $\mathcal{G}=(G_\lambda)_{\lambda\in\N}$ be a sequence of graphs with $|V(G_\lambda)|$, $\rho(G_\lambda) \to\infty$ as $\lambda\to\infty$. Define a function $t_{\mathcal{G}}(\lambda)=\rho(G_\lambda)/\ln \rho(G_\lambda)$.  If \begin{equation}\label{eq:weak_threshold_limit}\lim_{\lambda \to \infty} \frac{\ln \rho(G_\lambda)}{\ln \ln |V(G_\lambda)|} \,=\, \infty,\end{equation}
   then $t_{\mathcal{G}}(\lambda)$ 
   is a DP-threshold function for $\mathcal{G}$. Furthermore, if \begin{equation}\label{eq:threshold_limit}\lim_{\lambda \to \infty} \frac{\ln \rho(G_\lambda)}{\ln |V(G_\lambda)|} \,=\, 1,\end{equation}
            then $t_{\mathcal{G}}(\lambda)$ 
            is a sharp DP-threshold function for $\mathcal{G}$.
		\end{thm}

		\begin{proof}
			Assume \eqref{eq:weak_threshold_limit}. Set $n_\lambda = |V(G_\lambda)|$ and $\rho_\lambda = \rho(G_\lambda)$. Since $n_\lambda \to \infty$, \eqref{eq:weak_threshold_limit} implies that $\rho_\lambda \to \infty$ as well.
   %
%
 %
            Take any sequence $\kappa=(k_\lambda)_{\lambda\in\N}$ of positive integers satisfying $k_\lambda\leq t_\mathcal{G}(\lambda)$. Given $\epsilon > 0$, for all large enough $\lambda \in \N$, we have $\rho_\lambda \geq \exp(e/\epsilon)$ and $k_\lambda \leq \rho_\lambda/\ln\rho_\lambda$. Therefore, by Proposition~\ref{prop:first}, $\mathbb{P}(\text{$G_\lambda$ is $\mathcal{H}(G_\lambda,k_\lambda)$-colorable}) \leq \epsilon$ for all large enough $\lambda$, so $\mathcal{G}$ is non-$\kappa$-DP-colorable w.h.p. 
            Now suppose a sequence $\kappa=(k_\lambda)_{\lambda\in\N}$ satisfies $k_\lambda/t_\mathcal{G}(\lambda) \to \infty$. Let $d_\lambda$ be the degeneracy of $G_\lambda$ and note that $\rho_\lambda \leq d_\lambda \leq 2\rho_\lambda$. Take any $\epsilon \in (0,1/2)$. By \eqref{eq:weak_threshold_limit}, $d_\lambda \geq \rho_\lambda \geq \ln^{2/\epsilon} n$ for all large enough $\lambda$. Since
            \[
                \lim_{\lambda \to \infty} \left. k_\lambda \middle/ \frac{d_\lambda}{\ln d_\lambda} \right. \,=\, \lim_{\lambda \to \infty} \left. k_\lambda \middle/ \frac{\rho_\lambda}{\ln \rho_\lambda} \right. \,=\, \infty,
            \]
            we can apply Theorem~\ref{Degeneracy} to conclude that $\mathbb{P}(\text{$G_\lambda$ is $\mathcal{H}(G_\lambda,k_\lambda)$-colorable}) \geq 1 - \epsilon$. As $\epsilon$ is arbitrary, it follows that $\mathcal{G}$ is $\kappa$-DP-colorable w.h.p.
                
             To prove the ``furthermore'' part, assume that \eqref{eq:threshold_limit} holds. Take any $\epsilon > 0$ and suppose a sequence $\kappa=(k_\lambda)_{\lambda\in\N}$ satisfies $k_\lambda \geq (1+\epsilon)t_\mathcal{G}(\lambda)$ for all large enough $\lambda$. Consider any $\epsilon' \in (0,\epsilon)$ and set $s_\lambda = 1 - \ln \rho_\lambda/ \ln n_\lambda$, so $\rho_\lambda = n_\lambda^{1 - s_\lambda}$. By \eqref{eq:threshold_limit}, $s_\lambda \to 0$ as $\lambda \to \infty$. Thus, for all large enough $\lambda$,
            \[
                1 + \frac{s_\lambda}{1-2s_\lambda} \,\leq\, \sqrt{1 + \epsilon'}.
            \]
            For every such $\lambda$, we may apply Theorem~\ref{thm:slidingsharpness} with $\epsilon'' = \sqrt{1 + \epsilon'} - 1 < \epsilon'$ in place of $\epsilon$ to conclude that $\mathbb{P}(\text{$G_\lambda$ is $\mathcal{H}(G_\lambda,k_\lambda)$-colorable}) \geq 1 - \epsilon'' \geq 1 - \epsilon'$. Since $\epsilon'$ is arbitrary, it follows that $\mathcal{G}$ is $\kappa$-DP-colorable w.h.p., and the proof is complete.
        %
		\end{proof}

        
        The following observation is immediate from the definitions: 
		\begin{obs}\label{obs:equivThresh}
			If $t(\lambda)$ is a sharp DP-threshold function for a sequence of graphs $\mathcal{G}=(G_\lambda)_{\lambda\in\N}$ and $t'(\lambda)$ is another function such that $\lim_{\lambda\to\infty}t'(\lambda)/t(\lambda)=1$, then $t'(\lambda)$ is also a sharp DP-threshold function for $\mathcal{G}$.
		\end{obs}

		We will now use this to find threshold functions for specific sequences of graphs.

		\begin{customcor}{\bf\ref{cor:Kn}}
			For $\mathcal{G}=(K_n)_{n\in\N}$, the sequence of complete graphs, $t_{\mathcal{G}}(n)=n/(2\ln n)$ is a sharp DP-threshold function.
		\end{customcor}
		\begin{proof}
			The maximum density of $K_n$ is $\rho_n = (n-1)/2$. 
			By Theorem~\ref{thm:threshold}, a sharp DP-threshold function for $\mathcal{G}$ is $\rho_n/\ln \rho_n = (n-1)/(2 \ln((n-1)/2))$. Since
			\begin{equation*}
				\lim_{n\to\infty}\left.\frac{n}{2\ln n}\middle/\frac{n-1}{2\ln\left(\frac{n-1}{2}\right)}\right. \,=\,1,
			\end{equation*}
			we see that $n/(2\ln n)$ is a sharp DP-threshold function for $\mathcal{G}$ by Observation~\ref{obs:equivThresh}. 
		\end{proof}


		\begin{customcor}{\bf\ref{cor:Knn...}}
			For $\mathcal{G}=(K_{m \times n})_{n\in\N}$ with constant $m\geq2$, the sequence of complete $m$-partite graphs with $n$ vertices in each part, $t_{\mathcal{G}}(n)=(m-1)n/(2\ln n)$ is a sharp DP-threshold function.
		\end{customcor}
		\begin{proof}
   The maximum density of $K_{m \times n}$ is $\rho_n = (m-1)n/2$. This is a linear function of $n$, so we may apply Theorem~\ref{thm:threshold} to conclude that 
%
a sharp DP-threshold function for $\mathcal{G}$ is
			\begin{equation*}
    \frac{\rho_n}{\ln\rho_n}\,=\,\frac{(m-1)n}{2\ln\left(\frac{(m-1)n}{2}\right)}. 
			\end{equation*}
			The desired result follows by Observation~\ref{obs:equivThresh} since 
			\begin{equation*}
				\lim_{n\to\infty}\left.\frac{(m-1)n}{2 \ln n}\middle/ \frac{(m-1)n}{2\ln\left(\frac{(m-1)n}{2}\right)} \right. \,=\,1. \qedhere
			\end{equation*}
		\end{proof}

        \medskip

        \noindent \textbf{Acknowledgment.} We are grateful to the anonymous referee for helpful comments.

    \printbibliography

\end{document}